\documentclass{article}

\usepackage{amsfonts,amsthm}
\usepackage{amssymb}
\usepackage{amsmath}

\usepackage{mathabx}

\usepackage{sseq}
\usepackage{rotating}
\usepackage{bbm}
\usepackage[all,cmtip]{xy}
\usepackage{amscd}
\numberwithin{equation}{subsection}
\usepackage{caption}
\usepackage{subcaption}

\usepackage{tikz}
\usetikzlibrary{arrows}
\usepackage{enumitem}

\usepackage{mathrsfs}
\usepackage{color}
\usepackage{xcolor}

\usepackage[latin1]{inputenc}
\usepackage{graphicx}
\usepackage{dsfont}
\usepackage{hyperref}

\newtheorem{theorem}{Theorem}[section]

\newtheorem{lemma}[theorem]{Lemma}
\newtheorem{proposition}[theorem]{Proposition}
\newtheorem{corollary}[theorem]{Corollary}

\theoremstyle{remark}
\newtheorem{remark}[theorem]{Remark}

\theoremstyle{definition}

\newtheorem{example}[theorem]{Example}

\theoremstyle{theorem}

\newtheorem{PropositionA}{Proposition}
\newtheorem*{maintheorem}{Main Theorem}

\newcommand{\ow}{\omega}
\newcommand{\p}{\partial}

\newcommand{\C}{{\mathbb{C}}}

\newcommand{\R}{{\mathbb{R}}}

\newcommand{\Z}{{\mathbb{Z}}}

\renewcommand{\epsilon}{\varepsilon}
\renewcommand{\theta}{\vartheta}

\DeclareMathOperator{\FS}{FS}

\DeclareMathOperator{\Cl}{Clif}

\DeclareMathOperator{\id}{id}

\DeclareMathOperator{\Fix}{Fix}

\usepackage{graphicx}
\usepackage{authblk}

\begin{document}
\title{Uniqueness of real Lagrangians up to cobordism}

\makeatletter
\newcommand{\subjclass}[2][2010]{%
  \let\@oldtitle\@title%
  \gdef\@title{\@oldtitle\footnotetext{#1 \emph{Mathematics Subject Classification.} #2.}}%
}
\newcommand{\keywords}[1]{%
  \let\@@oldtitle\@title%
  \gdef\@title{\@@oldtitle\footnotetext{\emph{Key words and phrases.} #1.}}%
}
\makeatother

\author
{
Joontae Kim
}
\subjclass{53D12, 57N70, 55M35}
\keywords{real Lagrangian submanifold, antisymplectic involution, cobordism, Smith theory}
\date{}

\setcounter{tocdepth}{2}
\numberwithin{equation}{section}
\maketitle

\begin{abstract}
We prove that a real Lagrangian submanifold in a closed symplectic manifold is unique up to cobordism. We then discuss the classification of real Lagrangians in $\C P^2$ and $S^2\times S^2$. In particular, we show that a real Lagrangian in $\C P^2$ is unique up to Hamiltonian isotopy and that a real Lagrangian in $S^2\times S^2$ is either Hamiltonian isotopic to the antidiagonal sphere or Lagrangian isotopic to the Clifford torus.
\end{abstract}


\section{Introduction}\label{sec: intro}

A Lagrangian submanifold $L$ in a symplectic manifold $(M,\ow)$ is called \emph{real} if there exists an \emph{antisymplectic involution} $R\colon M\to M$, i.e., $R^2=\id_M$ and $R^*\ow=-\ow$, such that
$$
L=\Fix(R)
$$
where $\Fix(R)=\{x\in M\mid R(x)=x\}$ is the fixed point set. The fixed point set of an antisymplectic involution is necessarily a Lagrangian submanifold of $M$ if it is \emph{nonempty}, see Section \ref{sec: smith}. If $M$ is boundaryless, then so is $\Fix(R)$.

Historically, antisymplectic involutions and real Lagrangians have appeared in different aspects of symplectic topology. The \emph{Birkhoff} (antisymplectic) involution on $T^*\R^2$,
$$
R_\text{Birk}(q_1,q_2,p_1,p_2)=(q_1,-q_2,-p_1,p_2)
$$
played an intriguing role in the study of the restricted three body problem as a discrete symmetry to describe \emph{symmetric periodic orbits}, see \cite{Birk}.
The \emph{Arnold--Givental conjecture} \cite[Theorem C]{Givental} says that a real Lagrangian submanifold $L$ in a closed symplectic manifold $(M,\ow)$ manifests a rigidity phenomenon, namely that it satisfies an Arnold-type estimate: if $L\  \pitchfork\ \phi_H(L)$ for some Hamiltonian diffeomorphism $\phi_H$, then
$$
\#\big(L\cap \phi_H(L) \big)\geqslant \sum_{j=0}^{\dim L} \dim H_j(L;\Z_2).
$$
For partially proved cases, we refer to \cite{Urs,Oh3} and the literatures therein.

In this paper, we explore the topology of real Lagrangians towards their uniqueness and classification. In contrast to monotone Lagrangian submanifolds, for real Lagrangian submanifolds no classification result seems to be known. The main result of this paper is the following uniqueness statement about real Lagrangians. 
\begin{maintheorem}\label{thm: theoremA}
Any two real Lagrangians in a closed symplectic manifold are smoothly cobordant.
\end{maintheorem}
This will be an immediate consequence of Theorem \ref{thm: realimplieszero}, which describes an identity for cobordism classes. To prove the main theorem, we employ the cobordism theory developed by Conner--Floyd \cite{CF}. 

We cannot in general expect uniqueness of real Lagrangians up to \emph{smooth isotopy} or even up to \emph{diffeomorphism} since there are two real Lagrangians in $S^2\times S^2$ which are not diffeomorphic. Indeed, these real Lagrangians are smoothly cobordant, but \emph{not Lagrangian cobordant}, see Section \ref{sec: S2xS2}. This shows that without further assumptions, our uniqueness result for real Lagrangians up to smooth cobordism is optimal. Notice that we do not claim the existence of real Lagrangians. As pointed out by Grigory Mikhalkin, a symplectic $K3$-surface with a generic (with respect to the $\Z$-lattice) cohomology class of the symplectic form does not admit an antisymplectic involution at all, and hence in particular has no real Lagrangian submanifold.\\

The classification of real Lagrangians in $\C P^2$ and in the monotone $S^2\times S^2$ follows quite easily from known results of the classification on monotone Lagrangian submanifolds, which are based on $J$-holomorphic curve theory. Complex projective space $\C P^2$ endowed with the Fubini--Study form $\omega_{\FS}$ contains
the canonical real Lagrangian $\R P^2$, and as we show in Section 4.1 this is the only one:
\begin{PropositionA}\label{prop: A}
Any real Lagrangian in $\C P^2$ is Hamiltonian isotopic to $\R P^2$.
\end{PropositionA}
Consider $S^2\times S^2$ equipped with the split symplectic form $\ow\oplus \ow$, where $\ow$ denotes a Euclidean area form on $S^2$. The \emph{antidiagonal sphere}
	$$\overline{\Delta}=\{(x,-x)\mid x\in S^2\}\subset S^2\times S^2$$
	where $x\mapsto -x$ denotes the antipodal map, is a real Lagrangian via the antisymplectic involution $(x,y)\mapsto (-y,-x)$. The \emph{Clifford torus}
	$$
	\mathbb{T}_{\Cl}=S^1\times S^1\subset S^2\times S^2
	$$ defined as the product of the equators in each $S^2$-factor is a real Lagrangian torus. The classification result is the following.
\begin{PropositionA}\label{prop: B}
Let $L$ be a real Lagrangian in $S^2\times S^2$. Then $L$ is either Hamiltonian isotopic to the antidiagonal sphere $\overline{\Delta}$, or Lagrangian isotopic to the Clifford torus $\mathbb{T}_{\Cl}$.
\end{PropositionA}
By a \emph{Lagrangian isotopy} we mean a smooth isotopy through Lagrangian submanifolds. If two Lagrangians in a symplectic manifold are Hamiltonian isotopic, then they are necessarily Lagrangian isotopic. In general, the converse is not true. There exists a monotone Lagrangian torus in $S^2\times S^2$ which is Lagrangian isotopic to $\mathbb{T}_{\Cl}$ while not Hamiltonian isotopic to $\mathbb{T}_{\Cl}$, see Remark~\ref{rem: s2s2}.

\subsubsection*{Organization of the paper}
The main results of the paper are stated in the introduction. In Section \ref{sec: sec2}, we give a quick review of the Thom cobordism ring and Smith theory, as the main tools. In Section \ref{sec: sec3proof}, we prove the main theorem and its related corollaries. Finally, in Section \ref{sec: classification} we discuss the classification of real Lagrangians in $\C P^2$ and $S^2\times S^2$.

\section{Topology of fixed point sets of involutions}\label{sec: sec2}
We recall some facts from algebraic topology that are relevant for the study of fixed point sets of involutions. Throughout this paper, all manifolds and maps are assumed to be smooth.
\subsection{The Thom cobordism ring}
We briefly explain the Thom cobordism ring and  its basic properties. We refer to \cite{CF,Wall} for more details.

 Two closed $n$-manifolds $X_1$ and $X_2$ are \emph{cobordant} if there is a manifold $Y$ (called a \emph{cobordism}) such that $\p Y= X_1\sqcup X_2$. We say that a manifold $X$ is \emph{null-cobordant} (or bounds a manifold)  if there is a manifold $Y$ such that $\p Y=X$. The cobordism relation is an equivalence relation on the class of closed manifolds of fixed dimension. The equivalence class of a closed manifold $X$ is denoted by $[X]$ and is called a 
 \emph{cobordism class} of $X$.
For $n\geqslant 0$ we abbreviate
$$
\mathfrak{N}_n=\Big\{[X]\mid \text{$X$ is a closed manifold of dimension $n$} \Big\}.
$$
An abelian group structure on $\mathfrak{N}_n$ is given by disjoint union
$$
[X_1]+[X_2]:=[X_1\sqcup X_2].
$$
Note that $X$ is null-cobordant if and only if $[X]=0$ in $\mathfrak{N}_n$.
We abbreviate by $\mathfrak{N}=\sum_{n\geqslant 0}\mathfrak{N}_n$ the direct sum of cobordism groups. Topological product defines a product structure on $\mathfrak{N}$,
$$
[X_1]\cdot [X_2]:=[X_1\times X_2].
$$
The ring $\mathfrak{N}$ is called the \emph{Thom cobordism ring}. Thom \cite{Thom} showed that $\mathfrak{N}$ is a polynomial algebra over $\Z_2$, with one generator $x_i$ in each dimension $i$ not of the form $2^j-1$ for $j\geqslant 1$, i.e.,
$$
\mathfrak{N}=\Z_2[x_2,x_4,x_5,x_6,x_8,\dots].
$$
Moreover, he showed that real projective spaces $\R P^{2n}$ represent generators $x_{2n}$ of even dimensions, and Dold \cite{Dold} explicitly constructed closed manifolds representing the  generators $x_{2n+1}$ of odd dimensions. The following is obvious from the structure of the Thom cobordism ring: if a closed manifold $X$ satisfies $[X]^2=0$, then $[X]=0$. This will be used in the proof of the main theorem.

\begin{remark}\
\begin{enumerate}[label=(\arabic*)]
	\item \label{rem: 1} By Wall \cite[Lemma 7]{Wall}, it is known that $[\C P^n]=[\R P^n]^2$.
	\item Two closed manifolds are cobordant if and only if they have the same Stiefel--Whitney numbers, see \cite{Wall}.
	\item \label{rem: 2} Since the Stiefel--Whitney numbers of $\R P^{2n+1}$ vanish, $[\R P^{2n+1}]=0$ for all $n\geqslant 0$.
	\item \ref{rem: 1} and \ref{rem: 2} yield $[\C P^{2n}]=[\R P^{2n}]^2=x_{2n}^2\ne 0$ while $[\C P^{2n+1}]=0$.
	\item Let $X_1$ and $X_2$ be closed $n$-manifolds. Then
$$
[X_1]+[X_2]=[X_1\# X _2]
$$
where $\#$ denotes the connected sum of manifolds. This follows from the fact that the connected sum can be seen as a 0-surgery on $X_1\sqcup X_2$.
\end{enumerate}
\end{remark}

\subsection{Smith theory}\label{sec: smith}
The following theorem gives a topological constraint for the fixed point set of an involution. See \cite[Theorem 4.3, Chapter III]{Bre} for a proof.
\begin{theorem}[Euler characteristic relation]\label{thm: eulobs} Let $I\colon X\to X$ be an involution on a manifold $X$. Then we have
	$$
	\chi(X)=\chi(\Fix(I))\mod 2
	$$
	where $\chi$ denotes the Euler characteristic.
\end{theorem}
Working mod 2 in this theorem is necessary as the example $S^2$ with the equator given by the fixed point set of a reflection illustrates.
An immediate corollary is that every fixed point set of an involution on a manifold of odd Euler characteristic is nonempty. In particular, we obtain
\begin{corollary}\label{cor: oddeulerfixedpoint}
Every fixed point set of an antisymplectic involution on a symplectic manifold of odd Euler characteristic is nonempty and hence a real Lagrangian submanifold.
\end{corollary}
One should be aware that this is a non-symplectic phenomenon. As we mentioned in the introduction, the fixed point set of an antisymplectic involution $R$ on a symplectic manifold $(M,\ow)$ is Lagrangian if it is nonempty. This follows from the fact that the linear antisymplectic involution $T_xR\colon T_xM \to T_xM$ at $x\in L$ gives rise to a Lagrangian splitting $T_x M = V_1 \oplus V_{-1}$, where $V_{\pm 1}$ denotes the eigenspace of $T_xR$ to the eigenvalue $\pm 1$, and that $T_x\Fix(R) = V_1$.

Here are examples of symplectic manifolds of even Euler characteristic which admit an antisymplectic involution with and without fixed points.
\begin{example}\label{ex: nonemptyCPodd}
The antisymplectic involution on $\C P^{2n+1}$ defined by
	$$
	[z_0:\dots:z_{2n+1}]\mapsto [-\bar{z}_1:\bar{z}_0:\dots:-\bar{z}_{2n+1}:\bar{z}_{2n}]
	$$
	has no fixed point. On a closed orientable surface $\Sigma_g$ of genus $g \geq 0$ one easily finds antisymplectic involutions without and with fixed points as follows. 
	
Let $\tau$ denotes the rotation by $\pi$ on $S^2$ and $\rho$ the reflection on $S^2$ as in Figure~\ref{fig: involutions}. The reflection $\rho$ is an antisymplectic involution whose fixed point set is the equator. Since $\tau$ and $\rho$ commute, the composition $\tau\circ \rho$ is an antisymplectic involution. Indeed, it is the  antipodal map on $S^2$ and hence has no fixed point. The same idea applies for any closed orientable surfaces $\Sigma_g$.
\end{example}
	\begin{figure*}[h]
\begin{center}
\begin{tikzpicture}[scale=0.65]
\draw (0,0) circle (1.5cm);
\draw [dashed] (1.5, 0) arc [x radius = 1.5cm, y radius = 0.7cm, start angle = 0, end angle = 180];
\draw (1.5, 0) arc [x radius = 1.5cm, y radius = 0.7cm, start angle = 0, end angle = -180];


\draw [dashed] (0,2.5)--(0,-2) ;
\draw [<-] (0.4, 2) arc [x radius = 0.4cm, y radius = 0.2cm, start angle = -20, end angle = -200];
\draw [<->] (-2,-0.4) -- (-2, 0.4);
\node at (0.6,2.2) {$\tau$};
\node at (-1.8,0.7) {$\rho$};

\begin{scope}[xshift=3cm]
\draw (8, 0.1) arc [x radius = 0.7cm, y radius = 0.3cm, start angle = 0, end angle = -180];
\draw (5.4, 0.1) arc [x radius = 0.7cm, y radius = 0.3cm,
  start angle = 0, end angle = -180];
\draw (7.9, -0.05) arc [x radius = 0.6cm, y radius = 0.3cm,
  start angle = 0, end angle = 180];
\draw (5.3, -0.05) arc [x radius = 0.6cm, y radius = 0.3cm,
  start angle = 0, end angle = 180];
\draw (6,0) ellipse (3cm and 1.5cm);
\draw [dashed] (6,2.5)--(6,-2) ;
\draw [<-] (6.4, 2) arc [x radius = 0.4cm, y radius = 0.2cm, start angle = -20, end angle = -200];
\draw [<->] (2.6,-0.4) -- (2.6, 0.4);
\node at (6.6,2.2) {$\tau$};
\node at (2.8,0.7) {$\rho$};
\end{scope}
\end{tikzpicture}		
\end{center}
\caption{Antisymplectic involutions on $S^2$ and $\Sigma_2$.}
\label{fig: involutions}
\end{figure*}
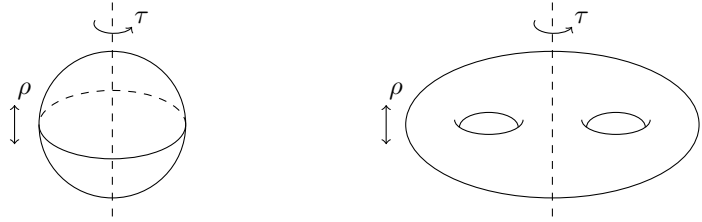
The next topological obstruction for being the fixed point set of an involution is the so-called \emph{Smith inequality}. It says that the topology of the fixed point set of an involution cannot be more complicated than the topology of the ambient manifold in terms of the sum of Betti numbers. See \cite[Theorem 4.1, Chapter~III]{Bre} for a proof.
\begin{theorem}[The Smith inequality]\label{thm: smithineq}
	Let $I\colon X\to X$ be an involution on a manifold $X$. Then we have
	$$
	\dim H_*(X;\Z_2)\geqslant \dim H_*(\Fix(I);\Z_2).
	$$
\end{theorem}
Using this inequality, we deduce that a submanifold of $\C P^n$ diffeomorphic to the $n$-torus cannot be the fixed point set of an involution, as $\dim H_*(\C P^n;\Z_2)=n+1<\dim H_*(T^n;\Z_2)=2^n$ for $n\geqslant 2$. In particular, the Clifford torus
$$
\mathbb{T}_{\Cl}^n=\big\{[z_0:\dots :z_n]\mid |z_0|=\cdots=|z_n| \big\}\subset \C P^n
$$
is not a real Lagrangian for $n\geqslant 2$. If $n=1$ the equator in $\C P^1=S^2$, which can be seen as the 1-dimensional Clifford torus $\mathbb{T}_{\Cl}^1$, is a real Lagrangian.

On the other hand, by Darboux's theorem, any symplectic manifold admits a Lagrangian torus. In contrast to this, the existence of a (closed) real Lagrangian in a symplectic manifold seems to be rare as this is a completely global problem. Indeed, there is no \emph{closed} real Lagrangian in $(\R^{2n},\ow_0=\sum_{j=1}^n dx_j\wedge dy_j)$ by the Smith inequality. 
\begin{remark}
Historically, Smith theory studied the topology of the fixed point sets of transformations of finite period on certain topological spaces. The central philosophy is that the fixed point set inherits the same cohomological characteristics of the ambient space. As a favorite result, if a manifold $X$ is a mod 2 homology sphere, then so is the fixed point set of an involution on $X$. This is a direct consequence of Theorems \ref{thm: eulobs} and \ref{thm: smithineq}.
\end{remark}

\section{Proof of the main theorem}\label{sec: sec3proof}

The main results of this paper are obtained from the following. 
\begin{theorem}\label{thm: realimplieszero}
If $L$ is a real Lagrangian in a closed symplectic manifold $(M,\ow)$, then $[M]=[L\times L]$.
\end{theorem}
The following theorem is a crucial ingredient of the proof. It intuitively says that the cobordism class of the ambient manifold is completely determined by a normal bundle of the fixed point set of an involution.
\begin{theorem}[Conner--Floyd]\label{thm: cruciallem}
	Let $I\colon X\to X$ be an involution on a closed manifold $X$ with nonempty fixed point set $F=\Fix(I)$. Then we have
	$$
	[X]=[{\bf P}(N_XF\oplus \epsilon)]
	$$
	where $N_XF$ is the normal bundle of $F$ in $X$, $\epsilon$ is the trivial real line bundle over $F$, and ${\bf P}(\cdot)$ denotes the real projectivization of a vector bundle.
\end{theorem}
We refer to \cite[Theorem {(24.2)}]{CF} for the proof. We are now in a position to prove Theorem \ref{thm: realimplieszero}.
\begin{proof}[Proof of Theorem \ref{thm: realimplieszero}]
	The result will follow from the fact that
	\begin{itemize}
		\item $L$ is the fixed point set of a smooth involution, and
		\item $N_ML$ is bundle isomorphic to $TL$ since $L$ is Lagrangian.
	\end{itemize}
We consider the involution $(x,y)\in L\times L\mapsto (y,x)$ with fixed point set $\Delta=\{(x,x)\mid x\in L\}$. The correspondence
$$
TL\to N_{L\times L}\Delta\subset TL\oplus TL,\quad v\mapsto (v,-v)
$$
defines a canonical bundle isomorphism. By Theorem \ref{thm: cruciallem}, we obtain
	$$
	[L\times L]=[{\bf P}(N_{L\times L}\Delta\oplus \epsilon)]=[{\bf P}(TL\oplus \epsilon)].
	$$
On the other hand, applying Theorem \ref{thm: cruciallem}  to $M$ and $L$, and since $TL\cong N_ML$ via an $\ow$-compatible almost complex structure $J$ on $M$, we see that
	$$
	[M]=[{\bf P}(N_ML\oplus \epsilon)]=[{\bf P}(TL\oplus \epsilon)].
	$$
It follows that $[M]=[L\times L]$.
\end{proof}
\begin{proof}[Proof of Main Theorem]
	Let $L_0$ and $L_1$ be two real Lagrangians in $M$. Then, by Theorem \ref{thm: realimplieszero}, we obtain $[L_0]^2=[L_1]^2$, which implies that $\big([L_0]+[L_1]\big)^2=0$. Hence, $[L_0]=[L_1]$.
\end{proof}
\begin{remark}\
\begin{itemize}
	\item The cobordism class of a real Lagrangian in a closed symplectic manifold $(M,\ow)$ is independent of the symplectic structure $\ow$ on $M$. More precisely, if $L_j$ is a real Lagrangian in $(M, \ow_j)$ for $j=1,2$, then $[L_1]=[L_2]$.
	\item It should be mentioned that Theorem \ref{thm: realimplieszero} also holds in \emph{equivariant cobordism}, see \cite[Section 28]{CF} for the definition of equivariant cobordism classes. More precisely, if $R\colon M\to M$ is an antisymplectic involution of a closed symplectic manifold $(M,\ow)$ with fixed point set $L=\Fix(R)$, then we have the identity for the equivariant cobordism classes 
	$$
	[M,R]=[L\times L, I],
	$$
	where $I(x,y)=(y,x)$ is the interchange involution on $L\times L$. This follows from \cite[Proposition 11]{Stong}. 

	\item 	The identity for the cobordism classes in Theorem \ref{thm: realimplieszero} also holds for the fixed point set of an antiholomorphic involution on a closed complex manifold, see  Conner--Floyd \cite[Theorem (24.4)]{CF}.
\end{itemize}
\end{remark}

\begin{example}\label{ex: notunique}
We cannot dispose of the assumption that the Lagrangian is real in the main theorem. Consider the two monotone Lagrangians $\R P^{2n}$ and $\mathbb{T}^{2n}_{\Cl}$ in $\C P^{2n}$. Since $[\R P^{2n}]\ne 0$ and $[\mathbb{T}^{2n}_{\Cl}]=0$, there is no cobordism between $\R P^{2n}$ and $\mathbb{T}^{2n}_{\Cl}$. Recall that the Clifford torus $\mathbb{T}_{\Cl}^n$ is not a real Lagrangian for $n\geqslant 2$.
\end{example}
By the main theorem, whenever we can specify a diffeomorphism type of a real Lagrangian in a closed symplectic manifold $(M,\ow)$, it determines the cobordism class of real Lagrangians in $M$. For example, $\R P^n$ is a real Lagrangian in $(\C P^n,\ow_{\FS})$ and the diagonal $\Delta=\{(x,x)\mid x\in M\}$ is a real Lagrangian in a product symplectic manifold $(M\times M,(-\ow)\oplus \ow)$, where $(M,\ow)$ is a closed symplectic manifold. We therefore obtain the following immediate corollaries of the main theorem.
\begin{corollary}\
\begin{itemize}
	\item Every real Lagrangian in $\C P^n$ is cobordant to $\R P^n$.
	\item Let $(M,\ow)$ be a closed symplectic manifold. Then every real Lagrangian in a symplectic manifold $(M\times M,(-\ow)\oplus\ow)$ is cobordant to $M$.	
\end{itemize}
\end{corollary}
It is worth noting that the mod 2 cohomology ring of a  real Lagrangian in $\C P^n$ is isomorphic to the mod 2 cohomology ring of $\R P^n$. Indeed, it is known by Bredon \cite[Theorem 3.2, Chapter VII]{Bre} that the fixed point set $F$ of any involution on $\C P^n$ satisfies one of the following:
\begin{enumerate}[label=(\arabic*)]
	\item $F$ is connected and $H^*(F;\Z_2)\cong H^*(\R P^n;\Z_2)$.
	\item $F$ is connected and $H^*(F;\Z_2)\cong H^*(\C P^n;\Z_2)$.
	\item $F=\emptyset$ and $n$ is odd.
	\item $F=F_1\sqcup F_2$ and $H^*(F_j;\Z_2)\cong H^*(\C P^{n_j};\Z_2)$ for $j=1,2$ and $n=n_1+n_2+1$.
\end{enumerate}
Here all isomorphisms have to be understood as  isomorphisms of graded rings. Since the only possible case for real Lagrangians is the first case, we obtain the following.

\begin{proposition}\label{prop: cohoRPinCP}
The mod 2 cohomology ring of a real Lagrangian in $\C P^n$ is isomorphic to the mod 2 cohomology ring of $\R P^n$.
\end{proposition}


Theorem \ref{thm: realimplieszero} also has the following corollary.
\begin{corollary}\label{cor: ambcobordant}
	Suppose that $L_j$ is a real Lagrangian in a closed symplectic manifold $(M_j,\ow_j)$ for $j=1,2$. If $L_1$ and $L_2$ are cobordant, then $M_1$ and $M_2$ are cobordant.	
\end{corollary}

\begin{example}
Corollary \ref{cor: ambcobordant} can fail if we drop the reality condition.
	Consider monotone Lagrangian tori in $S^2\times S^2$ and $\C P^2$. Since $[S^2\times S^2]=0$ and $[\C P^2]\ne 0$, there is no cobordism between $S^2\times S^2$ and $\C P^2$. 
\end{example}

\begin{example}
	The equator of $S^2$ is a real Lagrangian circle and the diagonal
	$
	\Delta=\{(x,x)\mid x\in S^1\}\subset T^2
	$
	is a real Lagrangian circle in $T^2$, and $S^2$ and $T^2$ are indeed cobordant.
\end{example}

\section{Classifications of real Lagrangians in some rational symplectic 4-manifolds}\label{sec: classification}
If a symplectic manifold $(M,\ow)$ is \emph{spherically monotone}, meaning that there exists $C>0$ such that for all $\alpha\in \pi_2(M)$ we have
$$
c_1(M)[\alpha] = C\cdot \ow(\alpha),
$$
then every real Lagrangian is monotone as a Lagrangian submanifold \cite[p.\ 954]{Oh}. By a \emph{monotone} Lagrangian submanifold we mean that there exists $K>0$ such that for all $\beta\in \pi_2(M,L)$ we have
$$
\mu_L(\beta)=K\cdot \ow(\beta)
$$
where $\mu_L\colon\pi_2(M,L)\to \Z$ denotes the Maslov class of $L$. We refer to \cite[Section~2]{Oh} for the definition of $\mu_L$. In that paper, $\mu_L$ is referred to as the Maslov index homomorphism $I_{\mu,L}$.

In this section, we discuss the classification problem of real Lagrangians in two spherically monotone rational symplectic manifolds: $(\C P^2,\ow_{\FS})$ and $(S^2\times S^2,\ow\oplus \ow)$, where $\ow$ is a Euclidean area form on $S^2$. \begin{remark}\
\begin{itemize}
	\item Since $c_1(\C P^n)=(n+1)[\ow_{\FS}]$ and $c_1(S^2\times S^2)=2[\ow\oplus\ow]$, the symplectic manifolds $\C P^2$ and $S^2\times S^2$ are spherically monotone. See \cite[Example 4.4.3]{MS}.
	\item A symplectic structure on $\C P^2$ and $S^2\times S^2$ which is spherically monotone is unique up to symplectomorphism and scaling, see \cite[Theorem B]{Tau}, \cite{Mcduff} and \cite[Theorem B]{LiLiu}.	
\end{itemize}
\end{remark}

\subsection{Real Lagrangians in $\C P^2$}\label{sec: cp2}
We have seen that a real Lagrangian in $\C P^n$ is cobordant to $\R P^n$ and that its mod 2 cohomology ring is isomorphic to the one of $\R P^n$. For $n=2$ a much stronger result holds, see Proposition \ref{prop: A}. We shall prove this.
\begin{proof}[Proof of Proposition \ref{prop: A}]
By the Smith inequality, the only closed surfaces that can be the fixed point set of an involution of $\C P^2$ are $S^2$ and $\R P^2$. The Euler characteristic relation further excludes $S^2$. In particular, a real Lagrangian in $\C P^2$ must be diffeomorphic to $\R P^2$. (This also follows directly from Proposition~\ref{prop: cohoRPinCP}.) We already know that $\R P^2$ is a real Lagrangian. By Li--Wu \cite[Theorem 6.9]{LiWu}, any two Lagrangian $\R P^2$ in $\C P^2$ are Hamiltonian isotopic.
\end{proof}
On the other hand, one can examine the \emph{Lagrangian cobordism relation} introduced by Arnold \cite{Arn2,Arn1}. We refer to Biran--Cornea's Lagrangian cobordism theory \cite{BCjams,BCfuk,BCcone} for recent developments.

We say that two closed Lagrangians $L$ and $L'$ in a symplectic manifold $(M,\ow)$ are \emph{Lagrangian cobordant} if there exist a smooth cobordism $V$ between $L$ and $L'$, and a Lagrangian embedding
$$
i\colon V\hookrightarrow \big(([0,1]\times \R)\times M,(dx\wedge dy)\oplus \ow \big)
$$
such that for some $\epsilon>0$ we have 
\begin{eqnarray*}
V\cap {\big([0,\epsilon)\times \R\times M\big)} &=& [0,\epsilon)\times \{0\}\times L	, \\
V\cap {\big((1-\epsilon,1]\times \R\times M\big)} &=& (1-\epsilon,1]\times \{0\}\times L'.
\end{eqnarray*}
If two connected Lagrangians are Hamiltonian isotopic, then the Lagrangian suspension construction shows that they are Lagrangian cobordant \cite[Section~2.3]{BCjams}. Hence, we obtain the following.
\begin{corollary}\label{cor: uniqueCP2}
	Any two real Lagrangians in $\C P^2$ are Lagrangian cobordant.
\end{corollary}
In general, two real Lagrangians may not be Lagrangian cobordant, see Section \ref{sec: S2xS2}.
\begin{remark}
The classification of monotone but not necessarily real Lagrangian submanifolds in $\C P^2$ is much more complicated:
\begin{itemize}
	\item By Shevchishin \cite{SheNotKtoCP2}, there is no Lagrangian Klein bottle in $\C P^2$.
	\item Vianna \cite[Theorem 1.1]{Vian} constructed infinitely many symplectomorphism classes of monotone Lagrangian tori in $\C P^2$ by (partially) proving the conjecture of Galkin--Usnich \cite[p.\ 4]{GalkinUsnich}. They conjectured the existence of an even larger class of infinitely many monotone Lagrangian tori in $\C P^2$.
	\item Actually, there is no Lagrangian sphere in $\C P^2$ for homological reasons: no homology class in $H_2(\C P^2;\Z)$ has self-intersection number $-2$.
	\item Besides for real Lagrangians, there are non-real monotone Lagrangians in $\C P^2$, for example, the Clifford torus $\mathbb{T}_{\Cl}^2$  and a Lagrangian submanifold in $\C P^2$ diffeomorphic to $\Sigma_2\# K$, where $\Sigma_2$ is a closed orientable surface of genus 2 and $K$ is the Klein bottle. See \cite[Theorem 2]{AG}.
\end{itemize}
\end{remark}

\subsection{Real Lagrangians in $S^2\times S^2$}\label{sec: S2xS2}
We shall prove Proposition \ref{prop: B}. By the Smith inequality,
$$
S^2, \quad T^2,\quad \R P^2,\quad K=\R P^2\#\R P^2,\quad S^2\sqcup S^2
$$
are the only closed surfaces that can possibly be the fixed point set of an involution. The Euler characteristic relation excludes $\R P^2$ from this list. Recall that the \emph{antidiagonal sphere} $\overline{\Delta}$ and the \emph{Clifford torus} $\mathbb{T}_{\Cl}$ are real Lagrangians in $S^2\times S^2$, see Section \ref{sec: intro}. By Weinstein neighborhood theorem, any Lagrangian sphere in a symplectic 4-manifold has self-intersection number $-2$ and hence is homologous to $\overline{\Delta}$. This implies that there is no disjoint union of two Lagrangian spheres in $S^2\times S^2$. We learned the following lemma and its proof from Grigory Mikhalkin. 
	\begin{lemma}\label{lem: mik}
Assume that the fixed point set $F = \Fix(I)$ of a smooth involution $I$ on $S^2 \times S^2$ is diffeomorphic to a closed connected surface. Then $F$ must be orientable.
	\end{lemma}
	\begin{proof}
Assume to the contrary that $F$ is non-orientable. We then find an embedded loop $\gamma$ in $F$ so that its mod 2 self-intersection number in $F$ is
$$
[\gamma]\cdot[\gamma]=1 \mod 2.
$$
Since $S^2\times S^2$ is simply-connected, we can choose a connected oriented surface $\Sigma \subset S^2\times S^2$ with $\p \Sigma = \gamma$ such that $\Sigma$ meets $F$ normally along $\gamma$ and transversely elsewhere.  Consider the closed surface $Z:=\Sigma\cup I(\Sigma)$ with homology class $[Z]\in H_2(S^2\times S^2;\Z_2)$. We claim that the mod 2 self-intersection number~of~$Z$~is
\begin{equation}\label{eq: eq1}
[Z]\cdot [Z] = [\gamma]\cdot [\gamma] = 1 \mod 2.
\end{equation}
But this is impossible since  the intersection form of $S^2\times S^2$ is even, i.e., $A\cdot A$ is even for all $A\in H_2(S^2\times S^2;\Z)$. 

To show \eqref{eq: eq1}, we shall construct an $I$-invariant perturbation $\tilde{Z}$ of $Z$ such~that
\begin{enumerate}
	\item [(i)] $\tilde{Z}$ and $Z$ intersect transversely,
	\item [(ii)] $\tilde{Z}\cap F$ is the union of an embedded loop $\tilde{\gamma}$ and finitely many points,
	\item [(iii)] $\tilde{\gamma}$ and $\gamma$ intersect transversely inside $F$,
	\item [(iv)] $Z\cap \tilde{Z}\cap F=\gamma\cap \tilde{\gamma}$.
\end{enumerate}
See the right drawing in Figure \ref{fig: construction}. The construction of $\tilde{Z}$ goes as follows. We first isotope $Z$ in a small neighborhood $U$ of $F$ in $M$ to an $I$-invariant surface $Z_1$ such that $Z_1$ and $Z$ intersect transversely in $U$, and the conditions (ii), (iii) and (iv) hold. Indeed, during this perturbation, $Z\cap F$ moves inside $F$. Since away from $\gamma$, the surface $Z$ intersects $F$ transversely in a finite number of points, we can achieve the condition (iv) by moving slightly near these points. Noting that away from $U$ the involution $I$ is free, we can isotope $Z_1$ (while keeping it fixed in a neighbourhood of $F$) to an $I$-invariant surface $\tilde{Z}$ such that all conditions (i)--(iv) are fulfilled. See Figure \ref{fig: construction} for a description of the construction.

Let $x\in Z\cap \tilde{Z}$ be an intersection point. Since $Z$ and $\tilde{Z}$ are $I$-invariant, we have $I(x)\in Z\cap \tilde{Z}$. If $I(x)=x$, then by condition (iv) we obtain $x\in \gamma\cap \tilde{\gamma}$. Hence, all intersection points of $Z$ and $\tilde{Z}$ come in pairs, except for those coming from $\gamma\cap \tilde{\gamma}$, whence \eqref{eq: eq1} follows, see Figure \ref{fig: construction}. 
\end{proof}
\begin{figure*}[h]
\begin{center}
\begin{tikzpicture}[scale=0.4]
\draw (-7,0)--(5,0);
\draw (-6,2)--(6,2);
\draw (-7,0)--(-6,2);
\draw (5,0)--(6,2);
\draw [red,thick] (-1,1) arc [x radius=1, y radius=0.4, start angle=0, end angle= -180];
\draw [red,thick,dashed] (-1,1) arc [x radius=1, y radius=0.4, start angle=0, end angle= 180];
\draw [thick] plot [smooth,tension=0.7] coordinates {
(-1,1)
(0,4)
(1,5)
(2,4)
(2.5,2)
(3,1)
 } ;

\draw [thick] plot [smooth,tension=0.7] coordinates {
(-3,1)
(-2,5)
(-1,7)
(0,7.5)
(1,7)
(2,6)
(3,4)
(3.5,2)
(3,1)
 } ;

\begin{scope}[yshift=2cm, yscale=-1]
\draw [thick] plot [smooth,tension=0.7] coordinates {
(-1,1)
(0,4)
(1,5)
(2,4)
(2.5,2)
(3,1)
 } ;

\draw [thick] plot [smooth,tension=0.7] coordinates {
(-3,1)
(-2,5)
(-1,7)
(0,7.5)
(1,7)
(2,6)
(3,4)
(3.5,2)
(3,1)
 } ;	
\end{scope}

\begin{scope}[xshift=0.5cm, yshift=2cm, yscale=-1]
\draw [dashed, thick] plot [smooth,tension=0.7] coordinates {
(1.7, 3.6)
(2.5,2)
(3,1)
 } ;
\draw [dashed, thick] plot [smooth,tension=0.7] coordinates {
(-1,1)
(-0.9,2)
(-0.7,3.5)
};
\draw [dashed, thick] plot [smooth,tension=0.7] coordinates {
(-3,1)
(-3,2)
(-2.9, 3.5)
};
\draw [dashed, thick] plot [smooth,tension=0.7] coordinates {
(2.7,3.6)
(3.5,2)
(3,1)
 } ;	
\end{scope}

\draw [red,fill] (3,1) circle [radius=0.1];

\begin{scope}[xshift=0.5cm, yscale=1]
\draw [dashed, thick] plot [smooth,tension=0.7] coordinates {
(1.7, 3.6)
(2.5,2)
(3,1)
 } ;
 \draw [dashed, thick] plot [smooth,tension=0.7] coordinates {
(-3,1)
(-3,2)
(-2.9, 3.5)
};
\draw [dashed, thick] plot [smooth,tension=0.7] coordinates {
(-1,1)
(-0.9,2)
(-0.7,3.5)
};
\draw [dashed, thick] plot [smooth,tension=0.7] coordinates {
(2.7,3.6)
(3.5,2)
(3,1)
 } ;	
  \draw [red, dashed] (-1,1) arc [x radius=1, y radius=0.4, start angle=0, end angle= -180];
\draw [red, dashed] (-1,1) arc [x radius=1, y radius=0.4, start angle=0, end angle= 180];
\end{scope}
\draw [red] (3.5,1) circle [radius=0.1];

\node at (-5,0.5) [above]{$F$};
\node at (3, 5.5) [above]{$Z$};
\node at (4.5,2.5) [above]{$Z_1$};
\node at (-5,0.5) [above]{$F$};
\node at (-3.5,0.5) [above]{$\gamma$};
\node at (0.1,0.5) [above]{$\tilde{\gamma}$};
\node at (-3,4) [above]{$\Sigma$};
\node at (-3,-4) [above]{$I(\Sigma)$};


\begin{scope}[xshift=16cm]
\draw (-7,0)--(5,0);
\draw (-6,2)--(6,2);
\draw (-7,0)--(-6,2);
\draw (5,0)--(6,2);
\draw [red,thick] (-1,1) arc [x radius=1, y radius=0.4, start angle=0, end angle= -180];
\draw [red,thick,dashed] (-1,1) arc [x radius=1, y radius=0.4, start angle=0, end angle= 180];
\draw [thick] plot [smooth,tension=0.7] coordinates {
(-1,1)
(0,4)
(1,5)
(2,4)
(2.5,2)
(3,1)
 } ;

\draw [thick] plot [smooth,tension=0.7] coordinates {
(-3,1)
(-2,5)
(-1,7)
(0,7.5)
(1,7)
(2,6)
(3,4)
(3.5,2)
(3,1)
 } ;

 \draw [blue,fill] (0.1,7.5) circle [radius=0.1];
 \draw [blue,fill] (1.2,4.95) circle [radius=0.1];
 \draw [blue,fill] (3.3,1.3) circle [radius=0.1];
\begin{scope}[yshift=2cm, yscale=-1]
\draw [thick] plot [smooth,tension=0.7] coordinates {
(-1,1)
(0,4)
(1,5)
(2,4)
(2.5,2)
(3,1)
 } ;

\draw [thick] plot [smooth,tension=0.7] coordinates {
(-3,1)
(-2,5)
(-1,7)
(0,7.5)
(1,7)
(2,6)
(3,4)
(3.5,2)
(3,1)
 } ;	
\end{scope}

\begin{scope}[xshift=0.5cm, yshift=2cm, yscale=-1]
\draw [dashed, thick] plot [smooth,tension=0.7] coordinates {
(-1,1)
(0,4)
(1,5)
(2,4)
(2.5,2)
(3,1)
 } ;

\draw [dashed, thick] plot [smooth,tension=0.7] coordinates {
(-3,1)
(-2,5)
(-1,7)
(0,7.5)
(1,7)
(2,6)
(3,4)
(3.5,2)
(3,1)
 } ;	
 \draw [blue,fill] (-0.4,7.5) circle [radius=0.1];
  \draw [blue,fill] (0.7,4.95) circle [radius=0.1];
 \draw [blue,fill] (2.8,1.3) circle [radius=0.1];
\end{scope}

\begin{scope}[xshift=0.5cm, yscale=1]
\draw [dashed, thick] plot [smooth,tension=0.7] coordinates {
(-1,1)
(0,4)
(1,5)
(2,4)
(2.5,2)
(3,1)
 } ;

\draw [dashed, thick] plot [smooth,tension=0.7] coordinates {
(-3,1)
(-2,5)
(-1,7)
(0,7.5)
(1,7)
(2,6)
(3,4)
(3.5,2)
(3,1)
 } ;	
  \draw [red, dashed] (-1,1) arc [x radius=1, y radius=0.4, start angle=0, end angle= -180];
\draw [red, dashed] (-1,1) arc [x radius=1, y radius=0.4, start angle=0, end angle= 180];

\end{scope}

\node at (-3,4) [above]{$Z$};
\node at (3,6) [above]{$\tilde{Z}$};
\node at (-3.5,0.5) [above]{$\gamma$};
\node at (0.1,0.5) [above]{$\tilde{\gamma}$};
\end{scope}

\draw [-implies,double equal sign distance] (6.5,1) -- (8.5,1);

\end{tikzpicture}		
\end{center}
\caption{The construction of $\tilde{Z}$: the intersection points of $Z$ and $\tilde{Z}$ away from $F$ come in pairs.}
\label{fig: construction}
\end{figure*}
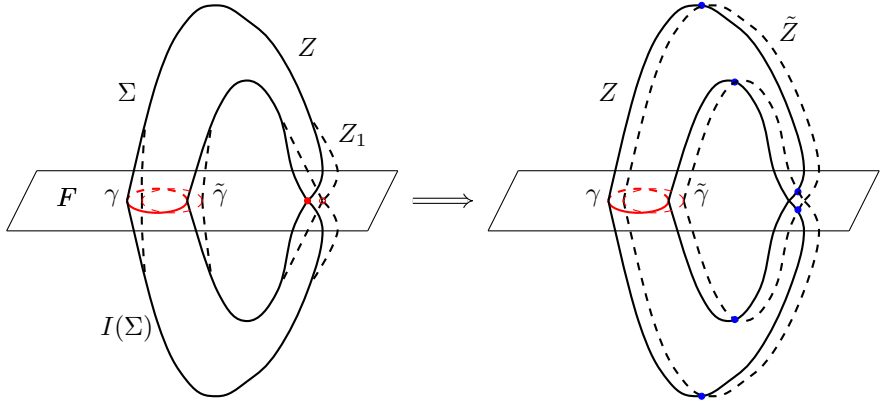

Hence, there is no a real Lagrangian Klein bottle in $S^2\times S^2$.	 By Hind \cite{HindS2S2}, any Lagrangian sphere in $S^2\times S^2$ is Hamiltonian isotopic to the antidiagonal~$\overline{\Delta}$. By Dimitroglou Rizell--Goodman--Ivrii \cite{DGI}, any two Lagrangian tori in $S^2\times S^2$ are Lagrangian isotopic to each other. To sum up, our discussion proves Proposition \ref{prop: B}.
	 
Although $\overline{\Delta}$ and $\mathbb{T}_{\Cl}$ are smoothly cobordant, they are not Lagrangian cobordant, which shows that the uniqueness of real Lagrangians up to \emph{Lagrangian cobordism} fails in general. This is an easy consequence of the following observation: if two closed Lagrangians in a symplectic manifold are Lagrangian cobordant, then they are \emph{$\Z_2$-homologous}, i.e., they represent the same $\Z_2$-homology classes. 

To see this, let $L$ and $L'$ be two Lagrangian submanifolds of the symplectic manifold $(M,\ow)$ that are Lagrangian cobordant through the Lagrangian embedding $i \colon V \hookrightarrow [0,1] \times \R \times M$.  Denoting by $\pi \colon[0,1]\times \R\times M\to M$ the projection onto $M$, the map $\pi\circ i\colon V\to M$ shows that $L$ and $L'$ are $\Z_2$-homologous. 

Recall that $H_2(S^2\times S^2;\Z_2)\cong \Z_2\oplus \Z_2$ is generated by the $S^2$-factors. By definition of $\overline{\Delta}$, its $\Z_2$-homology class is given by $(1,1)\in \Z_2\oplus \Z_2$. The $\Z$-homology class $[\mathbb{T}_{\Cl}]$ of $\mathbb{T}_{\Cl}$ vanishes and hence its $\Z_2$-homology class vanishes as well. Indeed, since the equator in $S^2$ represents the trivial class in $H_1(S^2;\Z)$, so does $\mathbb{T}_{\Cl}$ in $H_2(S^2\times S^2;\Z)$ by the K\"{u}nneth theorem. 
As a result, $\overline{\Delta}$ and $\mathbb{T}_{\Cl}$ are not $\Z_2$-homologous, and hence they are not Lagrangian cobordant.

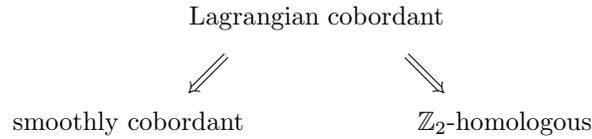
\begin{figure*}[h]
\begin{center}
\begin{tikzpicture}
\node at (0,0) {Lagrangian cobordant};
\node at (-2.5, -1.4) {smoothly cobordant};
\node at (2.5, -1.4) {$\Z_2$-homologous};
\draw [-implies,double equal sign distance] (-1.2,-0.5) -- (-1.7, -1);
\draw [-implies,double equal sign distance] (1.2,-0.5) -- (1.7, -1);
\end{tikzpicture}
\end{center}
\end{figure*}

\begin{remark}
The discussion above shows that the Lagrangians $L$ and $L'$ that are Lagrangian cobordant represent the same element in the \emph{singular cobordism group} $\mathfrak{N}_n(M)$, where $\dim L=n$. We refer to \cite[Section 4]{CF} for its definition. This is stronger than each of the following statements:
\begin{itemize}
	\item $L$ and $L'$ are smoothly cobordant,
	\item $L$ and $L'$ are $\Z_2$-homologous.
\end{itemize}
\end{remark}

\begin{remark}\label{rem: s2s2} The classification of monotone Lagrangian submanifolds in $S^2\times S^2$ has a rich history, specially for monotone Lagrangian tori.
\begin{itemize}
	\item By Chekanov--Schlenk \cite{ChekanovSchlenk,Chekanov}, there is a monotone Lagrangian torus, called the \emph{Chekanov torus}, in $S^2\times S^2$ which is not Hamiltonian isotopic to the Clifford torus.
	\item Vianna \cite[Theorem 1.1]{Vian} constructed infinitely many symplectomorphism classes of monotone Lagrangian tori in $S^2\times S^2$.
	\item Goodman \cite[Corollary 7]{Goodman} gave the construction of two Lagrangian Klein bottles in $S^2\times S^2$ which are not homologous to each other.
	\item By Abreu--Gadbled \cite[Theorem 2]{AG}, there exists a non-real monotone Lagrangian submanifold in $S^2\times S^2$ diffeomorphic to $\Sigma_4\# K$.
\end{itemize}	
\end{remark}

\subsection*{Acknowledgement}
The author cordially thanks Felix Schlenk for sharing insight with interest, Urs Frauenfelder and Jo\'e Brendel for fruitful discussions, and the Institut de Math\'ematiques at Neuch\^atel for its warm hospitality. Grigory Mikhalkin answered all my questions in the preprint and contributed them to  this paper. I specially thank Grigory Mikhalkin for his invaluable comments. This work is supported by a Swiss Government Excellence Scholarship.

\small{
School of Mathematics, Korea Institute for Advanced Study, 85 Hoegiro, Dongdaemun-gu, Seoul 02455, Republic of Korea\vspace{0.2cm} \\
\emph{E-mail address: }\texttt{joontae@kias.re.kr}
}

\end{document}